\documentclass[12pt,reqno]{amsart}
\usepackage{amsfonts}

\usepackage{amssymb,amsmath,graphicx,amsfonts,euscript}
\usepackage{color}
\usepackage{cite}

\usepackage[colorlinks=true]{hyperref}
\hypersetup{urlcolor=blue, citecolor=blue}

\newtheorem{theorem}{Theorem}[section]

\newtheorem{proposition}[theorem]{Proposition}
\newtheorem{definition}[theorem]{Definition}

\newtheorem{lemma}[theorem]{Lemma}

\newcommand{\R}{\mathbb{R}}
\newcommand{\Z}{\mathbb{Z}}
\newcommand{\F}{\dot{FB}}
\newcommand{\B}{\dot{B}}

\numberwithin{equation}{section}
\begin{document}
\title[Sharp well-posedness and ill-posedness for the micropolar fluid system]{Sharp well-posedness and ill-posedness
for the 3-D micropolar fluid system in Fourier-Besov spaces}

\author[Weipeng Zhu]{Weipeng Zhu$^{\text{1}}$}

\address{$^{\text{1}}$Department of Mathematics, Sun Yat-Sen University, Guangzhou 510275,  China}

\email{mathzwp2010@163.com.}

\vskip .2in
\begin{abstract}
  We study the Cauchy problem of the incompressible micropolar fluid system in $\mathbb{R}^{3}$. In a recent work of the first author and
  Jihong Zhao \cite{ZhuZ18}, it is proved that the Cauchy problem of the incompressible micropolar fluid system is locally well-posed in
  the Fourier--Besov spaces $\F^{2-\frac{3}{p}}_{p,r}$ for $1<p\leq\infty$ and $1\leq r<\infty$, and globally well-posed in these
  spaces with small initial data. In this work we consider the critical case $p=1$. We show that this problem is locally well-posed in
  $\F^{-1}_{1,r}$ for $1\leq r\leq 2$, and is globally well-posed in these spaces with small initial data. Furthermore, we prove
  that such problem is ill-posed in $\F^{-1}_{1,r}$ for $2< r\leq \infty$, which implies that the function space $\F^{-1}_{1,2}$ is
  sharp for well-posedness. In addition, using a similar argument we also prove that this problem is ill-posed in the Besov
  space $\B^{-1}_{\infty,r}$ with $2<r\leq\infty$.
\end{abstract}
\maketitle

\section{Introduction}
  In this paper we study the following initial value problem for the system of partial differential equations describing the motion
  of incompressible micropolar fluid:
\begin{equation}\label{eq1.1}
\begin{cases}
  \partial_{t} u-(\chi+\nu)\Delta u+ u\cdot\nabla u+\nabla\pi-2\chi\nabla\times\omega=0\;\; \mbox{in}\;\, \mathbb{R}^{3}\times\mathbb{R}_+, \\
  \partial_{t} \omega-\mu\Delta \omega + u\cdot\nabla\omega+4\chi\omega-\kappa\nabla\operatorname{div}\omega-2\chi\nabla\times u=0\;\;
  \mbox{in}\;\, \mathbb{R}^{3}\times\mathbb{R}_+, \\
  \operatorname{div} u=0\;\; \mbox{in}\;\, \mathbb{R}^{3}\times\mathbb{R}_+, \\
  (u,\omega)|_{t=0}=(u_0,\omega_0)\;\; \mbox{in}\;\, \mathbb{R}^{3}.
\end{cases}
\end{equation}
%---\eqref{eq1.1}---
  Here $u=u(x,t)$, $\omega=\omega(x,t)$ and $\pi=\pi(x,t)$ are unknown functions representing the linear velocity field,
  the micro-rotation velocity field and the pressure field of the fluid, respectively, and $\kappa$, $\mu$, $\nu$ and $\chi$ are positive constants
  reflecting various viscosity of the fluid. For simplicity, throughout this paper we only consider the situation with $\kappa=\mu=1$
  and $\chi=\nu={1}/{2}$.

  The model \eqref{eq1.1} was first theoretically studied by Eringen in the pioneering work \cite{E66}. It was proposed as an essential
  modification to the classical Navier-Stokes equations for the purpose to better describe the motion of various real world fluids
  consisting of rigid but randomly oriented particles (such as blood) by considering the effect of micro-rotation of the particles
  suspended in the fluid. A fluid possessing such a property is called a {\em micropolar fluid}, so that the model \eqref{eq1.1} is referred
  to as {\em micropolar fluid system} in the literature. Since the publication of the article of Eringen mentioned above, there have
  been some experiments in laboratory showing that solutions of the micropolar fluid system do better mimic behavior of real world
  fluids like blood than those of the classical Navier-Stokes equations, cf., \cite{MP84, P69, PRU74} and references therein.
  We refer the reader to see the references \cite{L03, L99} for more physical background of the above model.

  Clearly, if $\chi=0$ and $\omega=0$ then the system \eqref{eq1.1} reduces into the classical Navier-Stokes equations, which have been
  intensively studied during the past sixty years, especially during the past twenty years. We refer the reader to see the
  compositive books of Lemari\'{e}-Rieusset \cite{L02, L16} and references cited therein for interested reader on Navier-Stokes
  equations. Our interest of this paper is the case $\chi\not=0$ and $\omega\not=0$.

  Mathematical treatment of the micropolar fluid system \eqref{eq1.1} has also drawn much attention during the past fourty years.
  The first result on existence and uniqueness of solutions of the problem \eqref{eq1.1} was obtained by Galdi and Rionero in the reference
  \cite{GR77}. Existence of global weak solutions of the problem \eqref{eq1.1} was established by Lukaszewicz \cite{Luk89} and Boldrini and
  Rojas-Medar \cite{RojB98}. For existence and uniqueness of strong solutions to the problem \eqref{eq1.1} and more complex systems such
  as the magneto-micropolar fluid system, we refer the reader to see \cite{BolRF03,Roj97,RojO05}. Well-posedness of the problem \eqref{eq1.1}
  in various function spaces has also been well studied by many authors and some interesting results have been obtained. For instance, in
  \cite{FV07} Ferreira and Villamizar-Roa proved well-posedness of a more general model than \eqref{eq1.1} in pseudo-measure spaces.
  In \cite{CM12} Chen and Miao established global well-posedness of the problem \eqref{eq1.1} for small initial data in the
  Besov spaces $\dot{B}_{p,r}^{-1+\frac{3}{p}}(\mathbb{R}^3)$ for $p\in [1,6)$ and $r=\infty$. Moreover, if $r=1$, the range of $p$ for
  the existence can be extended to $[1,\infty)$. Recently, in a collaborating work of the first author of the present paper with Zhao
  \cite{ZhuZ18}, well-posedness of the problem \eqref{eq1.1} in the Fourier-Besov spaces $\F_{p,r}^{2-\frac{3}{p}}(\mathbb{R}^3)$ for
  $p\in (1,\infty]$ and $r\in [1,\infty)$ is established. We also refer the reader to see the references \cite{VilR08,WanW17,DonLW17,Yam05}
  for other related work.

  In this paper we study well-posedness of the problem \eqref{eq1.1} in the borderline Fourier-Besov spaces $\F_{1,r}^{-1}(\mathbb{R}^3)$
  ($1\leqslant r\leqslant\infty$). We prove that this problem is well-posed in $\F_{1,r}^{-1}(\mathbb{R}^3)$ for
  $1\leqslant r\leqslant 2$, while ill-posed for $2<r\leqslant\infty$. Further, as a by-product of the argument that we use
  to prove the second result, we also prove that the problem \eqref{eq1.1} is ill-posed in the Besov space $\dot{B}_{\infty,r}^{-1}(\mathbb{R}^3)$
  for $2<r\leqslant\infty$. For systematic study of the issue of sharp well-posedness and ill-posedness of the problem \eqref{eq1.1}
  in Besov-type spaces we leave for future work.

  Fourier-Besov spaces $\F_{p,r}^{s}(\mathbb{R}^3)$ ($s\in\mathbb{R}$, $p,r\in [1,\infty]$) (see Definition \ref{def1.2} below) were
  first introduced by Iwabuchi in \cite{Iwa11} in the study of the parabolic-elliptic Keller-Segel system. Later in \cite{IwaT14} Iwabuchi and Takada
  used such spaces to study the initial value problem of the Navier-Stokes-Coriolis system arising from geophysical fluid theory.
  They are introduced for the purpose to overcome difficulties caused by some linear terms like the terms $2\chi\nabla\times\omega$,
  $4\chi\omega$ and $2\chi\nabla\times u$ in the model \eqref{eq1.1}. Application of such spaces in the study of the classical Navier-Stokes
  equations was made by Konieczny and Yoneda in \cite{KY11}, where they proved global well-posedness of small-data initial value problem of
  the Navier-Stokes equations in critical Fourier-Besov spaces. We also refer the reader to see \cite{AFL17,FL14} for some extensions of these
  results in the Fourier-Besov-Morrey spaces for the active scalar equations and the Navier-Stokes-Coriolis equations, respectively.

  Before stating the main results of this paper, let us first introduce some notions and notations. Let $\mathcal{S}(\mathbb{R}^3)$
  be the Schwartz class of rapidly decreasing functions on $\mathbb{R}^3$, and $\mathcal{S}'(\mathbb{R}^3)$ the space of
tempered distributions. Moreover, we denote $\mathcal{S}'
_{h}(\mathbb{R}^3):=\mathcal{S}'(\mathbb{R}^3)/\mathcal{P}(\mathbb{R}^3)$,
where $\mathcal{P}(\mathbb{R}^3)$ is the set of polynomials (see
\cite{BCD11,T83}). Let $\varphi,\psi$ be two nonnegative functions
in $\mathcal{S}(\mathbb{R}^3)$ supported in $\mathcal{B}:=\{\xi \in
\mathbb{R}^3: \ |\xi| \leq \frac{4}{3}\}$ and $\mathcal{C}:=\{\xi
\in \mathbb{R}^3:\  \frac{3}{4} \leq |\xi| \leq \frac{8}{3}\}$,
respectively, such that
\begin{equation*}
  \sum_{j \in \mathbb{Z}}\psi( 2^{-j} \xi)=1, \quad \forall \xi \in \mathbb{R}^3 \backslash \{0\}
\end{equation*}
and
\begin{equation*}
  \varphi(\xi)+\sum_{j \geq 0} \psi(2^{-j}\xi)=1, \ \ \ \forall \xi \in \mathbb{R}^3.
\end{equation*}
For $u \in \mathcal{S}'(\mathbb{R}^3)$,  we define the homogeneous dyadic blocks
$\Delta_{j}$ and $S_{j}$  as follows:
\begin{equation*}
  \Delta_j u:=\mathcal{F}^{-1}(\psi(2^{-j}\xi)\hat{u}(\xi)) \ \ \  \text{and} \ \ \  S_j u:=\mathcal{F}^{-1}(\varphi(2^{-j}\xi)\hat{u}(\xi)),\ \ \ \forall j \in \mathbb{Z},
\end{equation*}
where $\mathcal{F}^{-1}$ is the inverse Fourier transform. Then for any $u\in \mathcal{S}' _{h}(\mathbb{R}^3)$, we have the
following well-known \textit{Littlewood-Paley decomposition}:
\begin{equation*}
  u=\sum_{j \in \mathbb{Z}}\Delta_j u \ \ \ \  \text{and} \ \ \ \ S_j u=\sum_{k\leq j-1} \Delta_k u.
\end{equation*}
Moreover, one easily checks that
\begin{equation*}
  \Delta_j \Delta_k u =0, \quad  |j-k| \geq 2 \ \ \  \text{and} \ \ \  \Delta_j(S_{k-1} u \Delta_k u)=0, \ \ \   |j-k| \geq 5.
\end{equation*}
\vskip.1in

The homogeneous Besov space $\dot{B}_{p,r}^s(\mathbb{R}^{3})$, the homogeneous Fourier-Besov space $\dot{FB}_{p,r}^s(\mathbb{R}^{3})$,
and the Chemin--Lerner type space $\tilde{L}^{\lambda}(0,T; \dot{FB}^{s}_{p,r}(\mathbb{R}^{3}))$ are respectively defined as follows:

\begin{definition}\label{def1.1}
Let $s\in \mathbb{R}$ and $1\leq p,r\leq\infty$. The space $\dot{B}^{s}_{p,r}(\mathbb{R}^{3})$ is defined to be the set of all tempered distributions $f\in \mathcal{S}'_{h}(\mathbb{R}^{3})$ such that the following norm is finite:
\begin{equation*}
  \|f\|_{\dot{B}^{s}_{p,r}}:= \begin{cases} \left(\sum_{j\in\mathbb{Z}}2^{jsr}\|\Delta_{j}f\|_{L^{p}}^{r}\right)^{\frac{1}{r}}
  \ \ &\text{if}\ \ 1\leq r<\infty,\\
  \sup_{j\in\mathbb{Z}}2^{js}\|\Delta_{j}f\|_{L^{p}}\ \
  &\text{if}\ \
  r=\infty.
 \end{cases}
\end{equation*}
\end{definition}

\begin{definition}\label{def1.2}
Let $s\in \mathbb{R}$ and $1\leq p,r\leq\infty$. The space $\dot{FB}^{s}_{p,r}(\mathbb{R}^{3})$ is defined to be the set of all tempered distributions $f\in \mathcal{S}'_{h}(\mathbb{R}^{3})$ such that $\hat{f} \in L_{loc}^1(\mathbb{R}^3)$ and the following norm is finite:
\begin{equation*}
  \|f\|_{\dot{FB}^{s}_{p,r}}:= \begin{cases} \left(\sum_{j\in\mathbb{Z}}2^{jsr}\|\widehat{\Delta_{j}f}\|_{L^{p}}^{r}\right)^{\frac{1}{r}}
  \ \ &\text{if}\ \ 1\leq r<\infty,\\
  \sup_{j\in\mathbb{Z}}2^{js}\|\widehat{\Delta_{j}f}\|_{L^{p}}\ \
  &\text{if}\ \
  r=\infty.
 \end{cases}
\end{equation*}
\end{definition}

\begin{definition}\label{def1.3} For $0<T\leq\infty$, $s\in \mathbb{R}$ and $1\leq p, r, \lambda\leq\infty$, we set  {\rm{(}}with the usual convention if
$r=\infty${\rm{)}}:
$$
  \|f\|_{\tilde{L}^{\lambda}_{T}(\dot{FB}^{s}_{p,r})}:=\big(\sum_{j\in\mathbb{Z}}2^{jsr}\|\widehat{\Delta_{j}f}\|_{L^{\lambda}(0,T;
  L^{p})}^{r}\big)^{\frac{1}{r}}.
$$
We then define the space $\tilde{L}^{\lambda}(0,T; \dot{FB}^{s}_{p,r}(\mathbb{R}^{3}))$  as the set of temperate distributions $f$ over $(0,T)\times \mathbb{R}^{3}$ such that $\displaystyle\lim_{j\rightarrow -\infty}S_{j}f=0$ in $\mathcal{S}'((0,T)\times\mathbb{R}^{3})$ and $\|f\|_{\tilde{L}^{\lambda}_{T}(\dot{FB}^{s}_{p,r})}<\infty$.
\end{definition}

  The corresponding linear system of the nonlinear system \eqref{eq1.1} is as follows:
\begin{equation}\label{eq1.2}
\begin{cases}
  \partial_{t} u-\Delta u-\nabla\times\omega=0,\\
  \partial_{t} \omega-\Delta \omega + 2\omega-\nabla\operatorname{div}\omega-\nabla\times u=0,\\
  \operatorname{div} u=0,\\
  (u,\omega)|_{t=0}=(u_0,\omega_0).
\end{cases}
\end{equation}
%---(1.2)---
  We use the notation $G(t)$ to denote the solution operator of the above problem, i.e., for given initial data $(u_0,\omega_0)$ in
  suitable function space, $(u,\omega)^T=G(t)(u_0,\omega_0)^T$ is the unique solution of the above problem. A simple computation shows
  that the operator $G(t)$ has the following expression:
\begin{align*}
(\widehat{G(t)f})(\xi)=e^{-\mathcal{A}(\xi)t}\hat{f}(\xi) \quad \mbox{for}\;\; f(x)=(f_1(x),f_2(x))^T,
\end{align*}
where
\begin{align*}
\mathcal{A}(\xi)=
\begin{bmatrix}
|\xi|^{2} I & \mathcal{B}(\xi) \\
\mathcal{B}(\xi) & (|\xi|^{2}+2)I+\mathcal{C}(\xi)
\end{bmatrix}
\end{align*}
with
\begin{align*}
\mathcal{B}(\xi)=i
\begin{bmatrix}
0 & \xi_3 & -\xi_2 \\
-\xi_3 & 0 &\xi_1 \\
\xi_2 & -\xi_1 &0
\end{bmatrix}
\ \ \text{and}\ \
\mathcal{C}(\xi)=
\begin{bmatrix}
{\xi_1}^2 & \xi_1\xi_2 & \xi_1\xi_3 \\
\xi_1\xi_2 & {\xi_2}^2 & \xi_2\xi_3 \\
\xi_1\xi_3 & \xi_2\xi_3 & {\xi_3}^2
\end{bmatrix}.
\end{align*}

On the other hand,  it is convenient to eliminate the pressure $\pi$ by applying the Leray projection $\mathbf{P}$ to both sides of the first equations of $\eqref{eq1.1}$, one has
\begin{equation}\label{eq1.3}
\begin{cases}
  \partial_{t} u-\Delta u+ \mathbf{P}(u\cdot\nabla u)-\nabla\times\omega=0,\\
  \partial_{t} \omega-\Delta \omega + u\cdot\nabla\omega+2\omega-\nabla\operatorname{div}\omega-\nabla\times u=0,\\
  \operatorname{div} u=0,\\
  (u,\omega)|_{t=0}=(u_0,\omega_0),
\end{cases}
\end{equation}
%---(1.4)---
where $\mathbf{P}:=I+\nabla(-\Delta)^{-1}\operatorname{div}$  is the $3\times 3$ matrix pseudo-differential operator in
$\mathbb{R}^3$ with the symbol
$(\delta_{ij}-\frac{\xi_i\xi_j}{|\xi|^2})_{i,j=1}^3$.
Denote by
\begin{align*}
U(x, t)=
\begin{pmatrix}
u(x, t) \\
\omega(x, t)
\end{pmatrix},\ \
U_0=
\begin{pmatrix}
u(x,0) \\
\omega(x,0)
\end{pmatrix}
=
\begin{pmatrix}
u_{0} \\
\omega_{0}
\end{pmatrix},\ \
U_i(x, t)=
\begin{pmatrix}
u_i(x, t) \\
\omega_i(x, t)
\end{pmatrix},\ \ i=1,2
\end{align*}
and
\begin{align*}
U_1 \tilde{\otimes} U_2=
\begin{pmatrix}
u_1 \otimes u_2 \\
u_1 \otimes \omega_2
\end{pmatrix},\ \
\ \
\tilde{\mathbf{P}}\nabla\cdot(U_1\tilde{\otimes} U_2)=
\begin{pmatrix}
\mathbf{P}\nabla\cdot(u_1\otimes u_2) \\
\nabla\cdot(u_1\otimes \omega_2)
\end{pmatrix}.
\end{align*}
Then by the Duhamel principle, the solution of system \eqref{eq1.3} can be reduced to finding a solution $U$ of the following integral equations:
\begin{equation}\label{eq1.4}
   U(t)=G(t)U_0-\int_0^t G(t-\tau) \tilde{\mathbf{P}}\nabla\cdot(U\tilde{\otimes} U)(\tau)d\tau.
\end{equation}
A solution of \eqref{eq1.4} is called a \textit{mild solution} of \eqref{eq1.1}.

The main results of this paper are as follows:
\begin{theorem}\label{th1.4}
Let $r\in [1,2]$ and $\alpha \in (0,1)$.
\begin{item}
\item[(1)]
For any initial data $(u_0,\omega_0)\in \dot{FB}^{-1}_{1,r}(\mathbb{R}^3)$ satisfying $\operatorname{div} u_0=0$, there exists a positive $T$ such that the system \eqref{eq1.1} has a unique mild solution such that
\begin{equation*}
  (u,\omega) \in C([0,T);\F^{-1}_{1,r}(\R^3)) \cap \tilde{L}^{\frac{2}{1+\alpha}}(0,T;\F^{\alpha}_{1,r}(\R^3)) \cap \tilde{L}^{\frac{2}{1-\alpha}}(0,T;\F^{-\alpha}_{1,r}(\R^3)).
\end{equation*}
\item[(2)]
There exists a positive constant $\epsilon$ such that for any initial data $(u_0,\omega_0)\in \dot{FB}^{-1}_{1,r}(\mathbb{R}^3)$ satisfying $\operatorname{div} u_0=0$ and
\begin{equation*}
  \|(u_0,\omega_0)\|_{\dot{FB}^{-1}_{1,r}}<\epsilon,
\end{equation*}
the system \eqref{eq1.1} has a unique global mild solution such that
\begin{equation*}
  (u,\omega) \in C([0,\infty);\F^{-1}_{1,r}(\R^3)) \cap \tilde{L}^{\frac{2}{1+\alpha}}(0,\infty;\F^{\alpha}_{1,r}(\R^3)) \cap \tilde{L}^{\frac{2}{1-\alpha}}(0,\infty;\F^{-\alpha}_{1,r}(\R^3)).
\end{equation*}
\end{item}
\end{theorem}

%\begin{remark}\label{re}
%Since
%
%\end{remark}

\begin{theorem}\label{th1.5}
For $r \in (2,\infty]$, the system \eqref{eq1.1} is ill-posedness in $\F^{-1}_{1,r}(\R^3)$ in the sense that the solution map is discontinuous at origin. More precisely, there exist a sequence of initial data $\{f^N\}_{N=1}^\infty \subset \F^{-1}_{1,2}(\R^3)$ and a sequence of time $\{t^N\}_{N=1}^\infty$ with
\begin{equation*}
  \|f^N\|_{\F^{-1}_{1,r}(\R^3)}\rightarrow 0  \ \mathrm{and}  \ t_N \rightarrow 0, \ \mathrm{as} \ N\rightarrow \infty,
\end{equation*}
such that the corresponding sequence of solutions $(u,\omega) \in C([0,\infty);\F^{-1}_{1,2}(\R^3))$ to initial value problem \eqref{eq1.1} with $(u,\omega)[f^N](0)=f^N$ satisfies
\begin{equation*}
  \|u[f^N](t^N)\|_{\F^{-1}_{1,r}(\R^3)} \geq c_0 \ \ \mathrm{and} \ \ \|\omega[f^N](t^N)\|_{\F^{-1}_{1,r}(\R^3)} \geq c_0,
\end{equation*}
with a positive constant $c_0$ independent of $N$.
\end{theorem}

\begin{theorem}\label{th1.6}
For $r \in (2,\infty]$, the system \eqref{eq1.1} is ill-posedness in $\B^{-1}_{\infty,r}(\R^3)$ in the sense that the solution map is discontinuous at origin. More precisely, there exist a sequence of initial data $\{f^N\}_{N=1}^\infty \subset \F^{-1}_{1,2}(\R^3)$ and a sequence of time $\{t^N\}_{N=1}^\infty$ with
\begin{equation*}
  \|f^N\|_{\B^{-1}_{\infty,r}(\R^3)}\rightarrow 0  \ \mathrm{and}  \ t_N \rightarrow 0, \ \mathrm{as} \ N\rightarrow \infty,
\end{equation*}
such that the corresponding sequence of solutions $(u,\omega) \in C([0,\infty);\F^{-1}_{1,2}(\R^3))$ to initial value problem \eqref{eq1.1} with $(u,\omega)[f^N](0)=f^N$ satisfies
\begin{equation*}
  \|u[f^N](t^N)\|_{\B^{-1}_{\infty,r}(\R^3)} \geq c_0 \ \ \mathrm{and} \ \ \|\omega[f^N](t^N)\|_{\B^{-1}_{\infty,r}(\R^3)} \geq c_0,
\end{equation*}
with a positive constant $c_0$ independent of $N$.
%For $r \in (2,\infty]$, the system \eqref{eq1.1} is ill-posedness in $\B^{-1}_{\infty,r}(\R^3)$.
\end{theorem}

%\begin{remark}\label{re1.4}
%By the boundedness of Fourier transform from $L^1(\R^3)$ to $L^\infty(\R^3)$, we have
%\begin{align}\label{eq1.5}
%\F^{-1}_{1,r}(\R^3) \hookrightarrow \B^{-1}_{\infty,r}(\R^3).
%\end{align}
%\end{remark}
\bigskip

The rest of this paper is organized as follows. In Section 2, we prove Theorem \ref{th1.4}. The proofs of Theorem \ref{th1.5} and Theorem \ref{th1.6} is given in Section 3.

\bigskip

\noindent \textbf{Notation.} In the following, let $c_1$ and $c_2$ be two positive constants, we use $A\sim B$ to denote $c_1 A\leq B\leq c_2 A$. Given $f\in L^1(\R^3)$, the Fourier transform $\mathcal{F}[f]$ (or $\hat{f}$) and the inverse Fourier transform $\mathcal{F}^{-1}[f]$ (or $\check{f}$) are defined by
$$\mathcal{F}[f](\xi)=\hat{f}(\xi):=\int_{\mathbb{R}^d}e^{-ix\xi}f(x)dx,\quad \mathcal{F}^{-1}[f](\xi)=\check{f}(\xi):=\frac{1}{(2\pi)^{d}}\int_{\mathbb{R}^d}e^{ix\xi}f(x)dx.$$

\section{The proof of Theorem \ref{th1.4}}
In this section, we aim at establishing the local and global existence and uniqueness of solutions for the generalized micropolar equations \eqref{eq1.1}. To this end, we establish some linear estimates for the semigroup $G(\cdot)$ and collect an important multiplication estimate.

\subsection{Linear estimates and product laws}

We first give the property of semigroup $G(\cdot)$.
\begin{lemma}\label{le2.1}
For $t\geq 0$ and $|\xi|\neq 0$. We have
\begin{equation}\label{eq2.1}
  \|e^{-t\mathcal{A}(\xi)}\| \leq e^{-|\xi|^{2}t} \ \ \mathrm{with}\ \ \|e^{-t\mathcal{A}(\xi)}\|=\sup_{\|f\|\leq 1} \|e^{-t\mathcal{A}(\xi)} f\|.
\end{equation}
Here $\|f\|=\max_{i}|a_i|$ with $\|f\|=\sum_{i=1}^{6}a_i v_i$, $v_1,v_2,\cdots,v_6$ are the eigenvectors for $\mathcal{A}(\xi)$.
\end{lemma}
\begin{proof}
See the proof in \cite{FV07}.
\end{proof}

Next, we establish the linear estimates for the semigroup $G(\cdot)$.
\begin{lemma}\label{le2.2}
Let $r \in [1,+\infty]$. Then there exists a positive constant $C$ such that
\begin{align*}
\|G(t)U_0\|_{\F^{-1}_{1,r}}\leq C\|U_0\|_{\F^{-1}_{1,r}}
\end{align*}
for all $t\geq 0$ and all $U_0\in \F^{-1}_{1,r}$.
\end{lemma}
\begin{proof}
By Lemma \ref{le2.1}, we have
\begin{align*}
\|G(t)U_0\|_{\F^{-1}_{1,r}}&=\left( \sum_{j\in\mathbb{Z}}2^{-jr}\|\mathcal{F}[G(t)\Delta_j U_0]\|_{L^1}^{r} \right)^{\frac{1}{r}} \\
&=\left( \sum_{j\in\mathbb{Z}}2^{-jr}\|e^{-t\mathcal{A}(\xi)} \mathcal{F}[\Delta_j U_0]\|_{L^1}^{r} \right)^{\frac{1}{r}} \\
&\leq C \left( \sum_{j\in\mathbb{Z}}2^{-jr} e^{-|\xi|^{2}rt} \|\mathcal{F}[\Delta_j U_0]\|_{L^1}^{r} \right)^{\frac{1}{r}} \\
&\leq C \left( \sum_{j\in\mathbb{Z}}2^{-jr} \|\mathcal{F}[\Delta_j U_0]\|_{L^1}^{r} \right)^{\frac{1}{r}} \\
&=C \|U_0\|_{\F^{-1}_{1,r}}.
\end{align*}
This completes the proof of Lemma \ref{le2.2}.
\end{proof}

\begin{lemma}\label{le2.3}
Let $r \in [1,+\infty]$, $T\in(0,+\infty]$ and $\alpha\in (0,1)$. Then there exists a positive constant $C=C(\alpha)$ depending only on $\alpha$ such that
\begin{align*}
\|G(t)U_0\|_{\tilde{L}^{\frac{2}{1\pm\alpha}}(0,T;\F^{\pm\alpha}_{1,r})}\leq C\|U_0\|_{\F^{-1}_{1,r}}
\end{align*}
for all $U_0\in \F^{-1}_{1,r}$.
\end{lemma}
\begin{proof}
By Definition \ref{def1.3}, it is easy to see that
\begin{align*}
   \|G(t)U_0\|_{\tilde{L}^{\frac{2}{1\pm\alpha}}(0,T;\F^{\pm\alpha}_{1,r})}   &=\Big(\sum_j 2^{jr(\pm\alpha)} \|\mathcal{F}[G(t)\Delta_j U_0]\|_{L^{\frac{2}{1\pm\alpha}}(0,T;L^1)}^{r} \Big)^{\frac{1}{r}}  \nonumber \\
   &=\Big(\sum_j 2^{jr(\pm\alpha)} \|e^{-t\mathcal{A}(\xi)} \mathcal{F}[\Delta_j U_0]\|_{L^{\frac{2}{1\pm\alpha}}(0,T;L^1)}^{r} \Big)^{\frac{1}{r}}  \nonumber \\
    & \leq C \Big(\sum_j 2^{jr(\pm\alpha)} \| e^{-t 2^{2j}}\|\mathcal{F}[\Delta_j U_0]\|_{L^1} \|_{L^{\frac{2}{1\pm\alpha}}(0,T)}^{r} \Big)^{\frac{1}{r}}   \nonumber \\
    & \leq C \Big(\sum_j 2^{jr(\pm\alpha)} 2^{-(1\pm\alpha) rj} \|\mathcal{F}[\Delta_j U_0]\|_{L^1}^r \Big)^{\frac{1}{r}}   \nonumber\\
    & \leq C \|U_0\|_{\dot{FB}^{-1}_{1,r}}.
\end{align*}
The proof of Lemma \ref{le2.3} is complete.
\end{proof}

\begin{lemma}\label{le2.4}
Let $T>0$, $s\in\R$ and $p,r,\lambda\in [1,\infty]$. Then there exists a positive constant $C$ such that
\begin{align*}
\Big\| \int_0^t G(t-\tau)f(\tau)d\tau \Big\|_{\tilde{L}^{\lambda}(0,T;\F^s_{p,r})} \leq C \|f\|_{\tilde{L}^{1}(0,T;\F^{s-\frac{2}{\lambda}}_{p,r})}
\end{align*}
for all $f\in \tilde{L}^{1}(0,T;\F^{s-\frac{2}{\lambda}}_{p,r})$.
\end{lemma}
\begin{proof}
By Young's inequality, we obtain
\begin{align*}
   &\Big\| \int_0^t G(t-\tau)f(\tau)d\tau \Big\|_{\tilde{L}^{\lambda}(0,T;\F^s_{p,r})} \nonumber \\
   & = \Big(\sum_j 2^{jrs} \Big\|\int_0^{t} e^{-(t-\tau) \mathcal{A}(\xi)} \mathcal{F}[\Delta_j f](\tau) d\tau \Big\|_{L^{\lambda}(0, T; L^p)}^{r} \Big)^{\frac{1}{r}}   \nonumber \\
   & \leq C \Big(\sum_j 2^{jrs} \Big\|\int_0^{t} e^{-(t-\tau) 2^{2 j}}\|\mathcal{F}[\Delta_j f](\tau)\|_{L^p}d\tau \Big\|_{L^{\lambda}(0, T)}^{r} \Big)^{\frac{1}{r}}   \nonumber \\
   & \leq C \Big(\sum_j 2^{jr(s-\frac{2}{\lambda})} \|\mathcal{F}[\Delta_j f](\tau)\|_{L^{1}(0, T; L^p)}^{r} \Big)^{\frac{1}{r}} \nonumber \\
   & \leq C \|f\|_{\tilde{L}^{1}(0,T;\F^{s-\frac{2}{\lambda}}_{p,r})}.
\end{align*}
We complete the proof of Lemma \ref{le2.4}.
\end{proof}

Finally, we collect an important multiplication estimate in the context of the homogeneous Fourier--Besov spaces.
\begin{lemma}\label{le2.5}
Let $r \in [1,2]$, $T\in(0,+\infty]$ and $\alpha\in (0,1)$. Then there exists a positive constant $C$ such that
\begin{align*}
\|fg\|_{\tilde{L}^{1}(0,T;\F^{0}_{1,r})}&\leq C\big(\|f\|_{\tilde{L}^{\frac{2}{1+\alpha}}(0,T;\F^{\alpha}_{1,r})}+\|g\|_{\tilde{L}^{\frac{2}{1-\alpha}}(0,T;\F^{-\alpha}_{1,r})}\nonumber\\
&\quad +\|g\|_{\tilde{L}^{\frac{2}{1+\alpha}}(0,T;\F^{\alpha}_{1,r})}+\|f\|_{\tilde{L}^{\frac{2}{1-\alpha}}(0,T;\F^{-\alpha}_{1,r})}\big).
\end{align*}
\end{lemma}
\begin{proof}
See Lemma 2.4 of \cite{SC15}.
\end{proof}

\subsection{The proof of Theorem \ref{th1.4}}
In order to prove Theorem \ref{th1.4}, we shall employ the following standard fixed point theorem.
\begin{proposition}\label{pro}
Let $\mathcal{X}$ be a Banach space, $\mathfrak{B}$ a continuous bilinear map from $\mathcal{X} \times \mathcal{X}$ to $\mathcal{X}$, and $\varepsilon$ a positive real number such that
\begin{equation*}
   \varepsilon < \frac{1}{4\|\mathfrak{B}\|} \ \ \text{with} \ \ \|\mathfrak{B}\|:=\sup_{\|u\|,\|v\| \leq 1}\|\mathfrak{B}(u,v)\|.
\end{equation*}
For any $y$ in the ball $B(0,\varepsilon)$ (i.e., with center $0$ and radius $\varepsilon$) in $\mathcal{X}$, then there exists a unique $x$ in $B(0,2\varepsilon)$ such that
\begin{equation*}
   x=y+\mathfrak{B}(x,x).
\end{equation*}
\end{proposition}
\begin{proof}
See Lemma 5.5 in \cite{BCD11}.
\end{proof}

\begin{proof}[Proof of Theorem \ref{th1.4} (1)]
Let $\alpha\in (0,1)$ be as in Theorem \ref{th1.4}. Given $T>0$, define the solution space $X^\alpha_T$ as
\begin{equation*}
  X^\alpha_T :=\big\{U:\ \ U \in \tilde{L}^{\frac{2}{1+\alpha}}(0,T;\F^{\alpha}_{1,r}) \cap \tilde{L}^{\frac{2}{1-\alpha}}(0,T;\F^{-\alpha}_{1,r})\big\}
\end{equation*}
and equipped with the following standard product norm:
\begin{equation*}
\|U\|_{X^\alpha_T}=\|U\|_{\tilde{L}^{\frac{2}{1+\alpha}}(0,T;\F^{\alpha}_{1,r})}+\|U\|_{\tilde{L}^{\frac{2}{1-\alpha}}(0,T;\F^{-\alpha}_{1,r})}.
\end{equation*}
Given $U\in
X^\alpha_T$, we define $\Phi(U)$ as follows
\begin{equation}\label{eq2.2}
   \Phi(U):=G(t)U_0- \int_0^t G(t-\tau) \tilde{\mathbf{P}}\nabla\cdot(U \tilde{\otimes} U)(\tau)d\tau.
\end{equation}
Obviously, $U$ is a mild solution of \eqref{eq1.1} on $[0,T]$ if and only if it
is a fixed point of $\Phi$.

We define the bilinear operator $B$ as
\begin{equation}\label{eq2.3}
   B(U_1,U_2):=\int_0^t G(t-\tau) \tilde{\mathbf{P}}\nabla\cdot(U_1\tilde{\otimes} U_2)(\tau)d\tau.
\end{equation}
Then by Lemma \ref{le2.4} and Lemma \ref{le2.5}, we have
\begin{align}\label{eq2.4}
  \|B(U_1,U_2)\|_{X^\alpha_T}&=\|\int_0^t G(t-\tau) \tilde{\mathbf{P}}\nabla\cdot(U_1\tilde{\otimes} U_2)(\tau)d\tau\|_{\tilde{L}^{\frac{2}{1+\alpha}}(0,T;\F^{\alpha}_{1,r})}\nonumber\\
  & \quad \quad +\|\int_0^t G(t-\tau) \tilde{\mathbf{P}}\nabla\cdot(U_1\tilde{\otimes} U_2)(\tau)d\tau\|_{\tilde{L}^{\frac{2}{1-\alpha}}(0,T;\F^{-\alpha}_{1,r})} \nonumber\\
  &\leq C \|\tilde{\mathbf{P}}\nabla\cdot(U_1\tilde{\otimes} U_2)\|_{\tilde{L}^{1}(0, \infty; \dot{FB}^{-1}_{1,r})} \nonumber\\
  &\leq C \Big\{ \|U_1\|_{\tilde{L}^{\frac{2}{1+\alpha}}(0,T;\F^{\alpha}_{1,r})} \|U_2\|_{\tilde{L}^{\frac{2}{1-\alpha}}(0,T;\F^{-\alpha}_{1,r})} \nonumber\\
  & \quad \quad + \|U_2\|_{\tilde{L}^{\frac{2}{1+\alpha}}(0,T;\F^{\alpha}_{1,r})} \|U_1\|_{\tilde{L}^{\frac{2}{1-\alpha}}(0,T;\F^{-\alpha}_{1,r})} \Big\}\nonumber\\
  & \leq  C_1 \|U_1\|_{X^\alpha_T} \|U_2\|_{X^\alpha_T}.
\end{align}
Combining \eqref{eq2.2} with \eqref{eq2.4}, we conclude that
\begin{equation*}\label{eq3.9}
   \|\Phi(U)\|_{X^\alpha_T} \leq \|G(t)U_0\|_{\dot{FB}^{-1}_{1,r}} +C_1\|U\|_{X^\alpha_T} \|U\|_{X^\alpha_T}.
\end{equation*}
Notice that applying Lemma \ref{le2.3}, we have
\begin{align}\label{eq2.5}
\|G(t)U_0\|_{X^\alpha_T}\leq C_2\|U_0\|_{\F^{-1}_{1,r}}.
\end{align}
Then $\|G(t)U_0\|_{X^\alpha_T} \rightarrow 0$ as $T \rightarrow 0$, since $\alpha \neq 1$. Hence, there exists $T>0$ such that $\|G(t)U_0\|_{X^\alpha_T}<\frac{1}{4C_1}$. Using Proposition \ref{pro}, system \eqref{eq1.1} admits a unique global mild solution $U\in X^\alpha_T$ with $\|U\|_{X^\alpha_T}<\frac{1}{2C_1}$. By using a standard density argument, we can further infer that $U \in C([0,T);\F^{-1}_{1,r}(\R^3))$. The proof of Theorem \ref{th1.4} (1) is complete.

Next, we replace $X^\alpha_T$ by $X^\alpha_\infty$. Then we have
\begin{align}\label{eq2.6}
  \|B(U_1,U_2)\|_{X^\alpha_\infty} \leq  C_1 \|U_1\|_{X^\alpha_\infty} \|U_2\|_{X^\alpha_\infty}.
\end{align}
Combining \eqref{eq2.5} with \eqref{eq2.6}, we conclude that
\begin{equation*}\label{eq3.9}
   \|\Phi(U)\|_{X^\alpha_\infty} \leq C_2\|U_0\|_{\F^{-1}_{1,r}} +C_1\|U\|_{X^\alpha_\infty} \|U\|_{X^\alpha_\infty}.
\end{equation*}
Hence, applying Proposition \ref{pro}, if $\|U_0\|_{\dot{FB}^{-1}_{1,r}}<\frac{1}{4C_1C_2}$, then system \eqref{eq1.1} admits a unique global mild solution $U\in X^\alpha_\infty$ with $\|U\|_{X^\alpha_\infty}<\frac{1}{2C_1}$.  This complete the proof of Theorem \ref{th1.4} (2).
\end{proof}

\section{Proofs of Theorem \ref{th1.5} and Theorem \ref{th1.6}}

  In this section we give the proofs of Theorem \ref{th1.5} and Theorem \ref{th1.6}.
  Before giving the proofs, let us sketch the ideas used in the proof of Theorem \ref{th1.5}, which can be used in Theorem \ref{th1.6} similarly. Motivated by \cite{BejT06,IwaT14}, we define the maps $A_n$ for $n=1,2,\cdots$ as follows:
\begin{equation*}
\begin{cases}
A_1(f):=G(t)f \\
A_n(f):=\displaystyle\sum_{n_1,n_2\geq 1,n_1+n_2=n} \int_0^t G(t-\tau) \tilde{\mathbf{P}}\nabla\cdot(A_{n_1}(f)\tilde{\otimes} A_{n_2}(f))(\tau)d\tau \quad \text{for} \quad n=2,3,\cdots.
\end{cases}
\end{equation*}
Then we construct a sequence $f^N$ such that $\|f^N\|_{\F^{-1}_{1,r}(\R^3)}\rightarrow 0$ when $N\rightarrow\infty$ and $\|f^N\|_{\F^{-1}_{1,2}(\R^3)} \leq C\delta$. For $r \in (2,\infty]$, it is easy to check that
\begin{align*}
\|A_1(f^N)\|_{\F^{-1}_{1,r}(\R^3)}\rightarrow 0,\quad N\rightarrow\infty,
\end{align*}
and
\begin{align*}
\sum_{k=3}^{\infty}\|A_k(f^N)\|_{\F^{-1}_{1,r}(\R^3)}\leq C\sum_{k=3}^{\infty}\|A_k(f^N)\|_{\F^{-1}_{1,r}(\R^3)} \leq C\delta^3, \quad N\gg 1,\ \delta\ll 1.
\end{align*}
The key point is to prove
\begin{align*}
\sum_{k=3}^{\infty}\|A_k(f^N)\|_{\F^{-1}_{1,r}(\R^3)}\geq  C\delta^2, \quad N\gg 1,\ \delta\ll 1,
\end{align*}
which implies that system \eqref{eq1.1} is ill-posedness in $\F^{-1}_{1,r}(\R^3)$.
\vskip .2in

{\em The proof of Theorem \ref{th1.5}}:\ \ We rewrite $\mathcal{A}(\xi)$ as follows:
\begin{align*}
\mathcal{A}(\xi)&=Q\operatorname{diag}(|\xi|^{2},2|\xi|^{2}+2,|\xi|^{2}-\sqrt{|\xi|^{2}+1}+1,|\xi|^{2}-\sqrt{|\xi|^{2}+1}+1,
\\&\quad \quad\quad |\xi|^{2}+\sqrt{|\xi|^{2}+1}+1,|\xi|^{2}+\sqrt{|\xi|^{2}+1}+1)Q^{-1}
\\&=Q\operatorname{diag}(|\xi|^{2},2|\xi|^{2},|\xi|^{2},|\xi|^{2},
 |\xi|^{2},|\xi|^{2})Q^{-1}
\\& +Q\operatorname{diag}(0,2,-\sqrt{|\xi|^{2}+1}+1,-\sqrt{|\xi|^{2}+1}+1,
\\&\quad \quad\quad \sqrt{|\xi|^{2}+1}+1,\sqrt{|\xi|^{2}+1}+1)Q^{-1}
\\& =: \mathcal{A}_1(\xi)+\mathcal{A}_2(\xi),
\end{align*}
where
%Here $Q$ is the $6\times 6$ invertible matrix.
\begin{align*}
Q=
\begin{bmatrix}
\begin{smallmatrix}
\frac{\xi_1}{\xi_3} & 0 & \frac{-\sqrt{-1}\xi_3 \tilde{\xi}^+}{|\xi|^2} & \frac{\sqrt{-1}\xi_2 \tilde{\xi}^+}{|\xi|^2} & \frac{\sqrt{-1}\xi_3 \tilde{\xi}^-}{|\xi|^2} & \frac{-\sqrt{-1}\xi_2 \tilde{\xi}^-}{|\xi|^2} \\
\frac{\xi_2}{\xi_3} & 0 & \frac{-\sqrt{-1}\xi_2\xi_3 \tilde{\xi}^+}{\xi_1|\xi|^2}  & \frac{-\sqrt{-1}(\xi_1^2+\xi_3^2) \tilde{\xi}^+}{\xi_1|\xi|^2} & \frac{\sqrt{-1}\xi_2\xi_3 \tilde{\xi}^-}{\xi_1|\xi|^2}  & \frac{\sqrt{-1}(\xi_1^2+\xi_3^2) \tilde{\xi}^-}{\xi_1|\xi|^2} \\
1 & 0 & \frac{\sqrt{-1}(\xi_1^2+\xi_2^2) \tilde{\xi}^+}{\xi_1|\xi|^2} & \frac{\sqrt{-1}\xi_2\xi_3 \tilde{\xi}^+}{\xi_1|\xi|^2} & \frac{-\sqrt{-1}(\xi_1^2+\xi_2^2) \tilde{\xi}^-}{\xi_1|\xi|^2} & \frac{-\sqrt{-1}\xi_2\xi_3 \tilde{\xi}^-}{\xi_1|\xi|^2} \\
0 & \frac{\xi_1}{\xi_3} & \frac{-\xi_2}{\xi_1} & \frac{-\xi_3}{\xi_1} & \frac{-\xi_2}{\xi_1} & \frac{-\xi_3}{\xi_1} \\
0 & \frac{\xi_2}{\xi_3} & 1 & 0 & 1 & 0\\
0 & 1 & 0 & 1 & 0 & 1
\end{smallmatrix}
\end{bmatrix};
\end{align*}
Here $\tilde{\xi}^+ = \sqrt{|\xi|^2+1}+1$ and $\tilde{\xi}^- = \sqrt{|\xi|^2+1}-1$. Since $\mathcal{A}_1(\xi)\mathcal{A}_2(\xi)=\mathcal{A}_2(\xi)\mathcal{A}_1(\xi)$, we have
\begin{align*}
e^{-t\mathcal{A}(\xi)}=e^{-t(\mathcal{A}_1(\xi)+\mathcal{A}_2(\xi))}=e^{-t\mathcal{A}_1(\xi)}e^{-t\mathcal{A}_2(\xi)}.
\end{align*}
Note that $e^{-t\mathcal{A}_2(\xi)}=\sum_{k=0}^{\infty} \frac{(-t\mathcal{A}_2(\xi))^k}{k!}=I+\sum_{k=1}^{\infty} \frac{(-t\mathcal{A}_2(\xi))^k}{k!}$,
then the semigroup $G(\cdot)$ satisfies
\begin{align*}
\widehat{G}(\xi,t)=e^{-t\mathcal{A}(\xi)}=e^{-t\mathcal{A}_1(\xi)}(I+\sum_{k=1}^{\infty} \frac{(-t\mathcal{A}_2(\xi))^k}{k!})=:\widehat{G_m}(\xi,t)+\widehat{G_r}(\xi,t),
\end{align*}
where
\begin{align*}
\widehat{G_m}(\xi,t)=e^{-t\mathcal{A}_1(\xi)}=
\begin{bmatrix}
e^{-t|\xi|^{2}} I & 0 \\
0 & R(\xi)
\end{bmatrix},
\end{align*}
with
\begin{align*}
R(\xi,t)=
\begin{bmatrix}
\begin{smallmatrix}
\frac{\xi_1^2 e^{-2t|\xi|^2}+\xi_2^2 e^{-t|\xi|^2}+\xi_3^2 e^{-t|\xi|^2}}{|\xi|^2} & \frac{\xi_1 \xi_2 e^{-2t|\xi|^2}-\xi_1 \xi_2 e^{-t|\xi|^2}}{|\xi|^2} & \frac{\xi_1 \xi_3 e^{-2t|\xi|^2}-\xi_1 \xi_3 e^{-t|\xi|^2}}{|\xi|^2} \\
\frac{\xi_1 \xi_2 e^{-2t|\xi|^2}-\xi_1 \xi_2 e^{-t|\xi|^2}}{|\xi|^2} & \frac{\xi_1^2 e^{-t|\xi|^2}+\xi_2^2 e^{-2t|\xi|^2}+\xi_3^2 e^{-t|\xi|^2}}{|\xi|^2} & \frac{\xi_2 \xi_3 e^{-2t|\xi|^2}-\xi_2 \xi_3 e^{-t|\xi|^2}}{|\xi|^2} \\
\frac{\xi_1 \xi_3 e^{-2t|\xi|^2}-\xi_1 \xi_3 e^{-t|\xi|^2}}{|\xi|^2} & \frac{\xi_2 \xi_3 e^{-2t|\xi|^2}-\xi_2 \xi_3 e^{-t|\xi|^2}}{|\xi|^2} & \frac{\xi_1^2 e^{-t|\xi|^2}+\xi_2^2 e^{-t|\xi|^2}+\xi_3^2 e^{-2t|\xi|^2}}{|\xi|^2}
\end{smallmatrix}
\end{bmatrix}.
\end{align*}
\bigskip

Let $f^N:=(u_0^N,\omega_0^N)$ and define
\begin{align*}
\chi(\xi)=\begin{cases}
  1, \quad\quad \text{if}\ |\xi_k|\leq 1,\ k=1,2,3\\
  0, \quad\quad \text{otherwise},
\end{cases}
\end{align*}
and $\chi_j^{\pm}=\chi(\xi\mp 2^j e_2)$ for $j\in \Z$, where $e_2=(0,1,0)$.
Then we construct a sequence $\{(u_0^N,\omega_0^N)\}_{N=1}^\infty$ by Fourier transform
\begin{align*}
\widehat{u_0^N}(\xi)=\frac{\delta\sqrt{-1}}{N^{\frac{1}{2}}}\sum_{j=N}^{[\frac{3}{2}N]+1}
2^j(\chi_j^{+}(\xi)+\chi_j^{-}(\xi)) \frac{1}{|\xi|} \begin{pmatrix}
                                                     \xi_2 \\
                                                     -\xi_1 \\
                                                     0
                                                   \end{pmatrix}
\end{align*}
and
\begin{align*}
\widehat{\omega_0^N}(\xi)=\frac{\delta\sqrt{-1}}{N^{\frac{1}{2}}}\sum_{j=N}^{[\frac{3}{2}N]+1}
2^j(\chi_j^{+}(\xi)+\chi_j^{-}(\xi))  \frac{1}{|\xi|}\begin{pmatrix} \xi_2 \\ 0 \\ 0\end{pmatrix},
\end{align*}
here $\delta$ is a small constant which will be chosen later. It is easy to see that
\begin{align*}
\|f^N\|_{\F^{-1}_{1,r}} \leq C\delta N^{\frac{1}{r}-\frac{1}{2}}
\end{align*}
for all $N\in \mathbb{N}$ and all $2\leq r\leq \infty$. Then by Lemma \ref{le2.2}, we obtain
\begin{align}\label{eq4.1}
\|A_1(f^N)\|_{\F^{-1}_{1,r}} \leq C\delta N^{\frac{1}{r}-\frac{1}{2}}
\end{align}
for all $N\in \mathbb{N}$ and all $2\leq r\leq \infty$.

Let $E$ be a measurable set in $\R^3$ such that the Lebesgue measure of $E$ is positive, there exists a constant $C>0$ such that
\begin{align*}
1-\frac{\xi_1^2}{|\xi|^2}\geq C,\quad \text{for}\ \xi=(\xi_1,\xi_2,\xi_3)\in E,
\end{align*}
and
\begin{align*}
E \subset \{\xi\in \R^3\mid \frac{1}{10}\leq \xi_1 \leq 1,\ |\xi|\leq 1\}.
\end{align*}
Since $\frac{1}{10}\leq |\xi| \leq 1$ for all $\xi\in E$, there exists $j_0\in \mathbb{N}$ such that $\sum_{j=-j_0}^{j_0}\widehat{\psi_j}(\xi)=1$ for all $\xi\in E$.

Note that
\begin{align*}
(u[f^N]^T,\omega[f^N]^T)^T =\sum_{k=1}^{\infty} A_k(f^N),
\end{align*}
define
\begin{align*}
(u_k(f^N)^T,\omega_k(f^N)^T)^T :=A_k(f^N), \quad k=1,2,\cdots,
\end{align*}
then
\begin{align*}
u[f^N]=\sum_{k=1}^{\infty} u_k(f^N),\quad \omega[f^N]=\sum_{k=1}^{\infty} \omega_k(f^N).
\end{align*}

By considering the first component of $\mathcal{F}[u_2(f^N)(t)](\xi)$, we have
\begin{align}\label{eq4.2}
&|\mathcal{F}[u_2(f^N)(t)](\xi)| \nonumber\\
&\geq \Big|\int_0^t e^{-(t-\tau)|\xi|^2} \sum_{l=1}^3 (\delta_{1,l}-\frac{\xi_1 \xi_l}{|\xi|^2}) \sum_{k=1}^3 \xi_k (\widehat{G(\tau)f^N})_k * (\widehat{G(\tau)f^N})_l d\tau\Big| \nonumber\\
&\ \ -\Big|\int_0^t |\widehat{G_r}(\xi,t-\tau)| \mathcal{F}[\tilde{\mathbf{P}} \nabla\cdot (G(\tau)f^N \tilde{\otimes} G(\tau)f^N)](\xi)d\tau\Big| \nonumber\\
&\geq \Big|\int_0^t e^{-(t-\tau)|\xi|^2} (1-\frac{\xi_1^2}{|\xi|^2}) \xi_1 (e^{-\tau|\xi|^2} \widehat{u_0^N})_1 * (e^{-\tau|\xi|^2} \widehat{u_0^N})_1 d\tau\Big| \nonumber\\
&\ \ -\Big|\int_0^t e^{-(t-\tau)|\xi|^2} (1-\frac{\xi_1^2}{|\xi|^2}) \xi_2 (e^{-\tau|\xi|^2} \widehat{u_0^N})_2 * (e^{-\tau|\xi|^2} \widehat{u_0^N})_1 d\tau\Big| \nonumber\\
&\ \ -\Big|\int_0^t e^{-(t-\tau)|\xi|^2} \frac{\xi_1 \xi_2}{|\xi|^2} \xi_1 (e^{-\tau|\xi|^2} \widehat{u_0^N})_1 * (e^{-\tau|\xi|^2} \widehat{u_0^N})_2 d\tau\Big| \nonumber\\
&\ \ -\Big|\int_0^t e^{-(t-\tau)|\xi|^2} \frac{\xi_1 \xi_2}{|\xi|^2} \xi_2 (e^{-\tau|\xi|^2} \widehat{u_0^N})_2 * (e^{-\tau|\xi|^2} \widehat{u_0^N})_2 d\tau\Big| \nonumber\\
&\ \ -2\sum_{k,l=1}^{3} \Big|\int_0^t e^{-(t-\tau)|\xi|^2}  (\delta_{1,l}-\frac{\xi_1 \xi_l}{|\xi|^2}) \xi_k (e^{-\tau|\xi|^2} \widehat{u_0^N})_k * (|\widehat{G_r}(\xi,\tau)| \widehat{f^N})_l d\tau\Big| \nonumber\\
&\ \ -\sum_{k,l=1}^{3} \Big|\int_0^t e^{-(t-\tau)|\xi|^2}  (\delta_{1,l}-\frac{\xi_1 \xi_l}{|\xi|^2}) \xi_k (|\widehat{G_r}(\xi,\tau)| \widehat{f^N})_k * (|\widehat{G_r}(\xi,\tau)| \widehat{f^N})_l d\tau\Big| \nonumber\\
&\ \ -\Big|\int_0^t |\widehat{G_r}(\xi,t-\tau)| \mathcal{F}[\tilde{\mathbf{P}} \nabla\cdot (G(\tau)f^N \tilde{\otimes} G(\tau)f^N)](\xi)d\tau\Big| \nonumber\\
& =: J_1(\xi,t)-J_2(\xi,t)-J_3(\xi,t)-J_4(\xi,t)-J_5(\xi,t)-J_6(\xi,t)-J_7(\xi,t).
\end{align}
We first estimates for $J_1(\xi,t)$ with $\xi \in E$. By the definitions of $u_0^N$, we have
\begin{align*}
J_1(\xi,t)= \Big|-2\int_0^t e^{-(t-\tau)|\xi|^2} &(1-\frac{\xi_1^2}{|\xi|^2}) \xi_1 \int_{\R^3} e^{-\tau(|\xi-\eta|^2+|\eta|^2)}
\\&\times \frac{(\xi_2-\eta_2)\eta_2}{|\xi-\eta||\eta|} \frac{\delta^2}{N} \sum_{j=N}^{[\frac{3}{2}N]+1} 2^{2j} \chi_j^{+}(\xi-\eta) \chi_j^{-}(\eta) d\eta d\tau\Big|.
\end{align*}
Since
\begin{align*}
 -1\leq \frac{(\xi_2-\eta_2)\eta_2}{|\xi-\eta||\eta|} \leq -\frac{1}{16}
\end{align*}
for all $\eta \in \operatorname{supp} \chi_j^{-}$ with $\xi-\eta \in \operatorname{supp} \chi_j^{+}$, or all $\eta \in \operatorname{supp} \chi_j^{+}$ with $\xi-\eta \in \operatorname{supp} \chi_j^{-}$, we obtain
\begin{align*}
J_1(\xi,t) &\geq \frac{C\delta^{2}}{N} \sum_{j=N}^{[\frac{3}{2}N]+1} 2^{2j} \int_0^t \int_{\R^3} e^{-\tau(|\xi-\eta|^2+|\eta|^2)} \chi_j^{+}(\xi-\eta) \chi_j^{-}(\eta) d\eta d\tau \\
&\geq \frac{C\delta^{2}}{N} \sum_{j=N}^{[\frac{3}{2}N]+1} 2^{2j} 2^{-2j} (1-e^{8 2^{2j}t}) \\
&\geq C\delta^{2}
\end{align*}
for all $\xi\in E$ and all $t\sim 2^{-2N}$, which yields that
\begin{align}\label{eq4.3}
\|J_1(\xi,t)\|_{L^1(E)}\geq C\delta^{2}.
\end{align}
On the estimates of $J_2(\xi,t)$ and $J_3(\xi,t)$ for $\xi\in E$, we have
\begin{align}\label{eq4.4}
J_2(\xi,t)+J_3(\xi,t) &\leq \frac{C\delta^{2}}{N} \sum_{j=N}^{[\frac{3}{2}N]+1} 2^{2j} \int_0^t \int_{\R^3} e^{-\tau(|\xi-\eta|^2+|\eta|^2)} \frac{|\xi_1-\eta_1| |\eta_2|}{|\xi-\eta||\eta|} \chi_j^{+}(\xi-\eta) \chi_j^{-}(\eta) d\eta d\tau \nonumber\\
&\leq \frac{C\delta^{2}}{N} \sum_{j=N}^{[\frac{3}{2}N]+1} 2^{2j} \int_0^t \int_{\R^3} e^{-\tau(|\xi-\eta|^2+|\eta|^2)} 2^{-j} \chi_j^{+}(\xi-\eta) \chi_j^{-}(\eta) d\eta d\tau \nonumber\\
&\leq \frac{C\delta^{2}}{N} \sum_{j=N}^{[\frac{3}{2}N]+1} 2^{2j} 2^{-2j} 2^{-j} \nonumber\\
&\leq \frac{C\delta^{2}}{N2^N}.
\end{align}
On the estimates of $J_4(\xi,t)$ for $\xi\in E$, we have
\begin{align}\label{eq4.5}
J_4(\xi,t) &\leq \frac{C\delta^{2}}{N} \sum_{j=N}^{[\frac{3}{2}N]+1} 2^{2j} \int_0^t \int_{\R^3} e^{-\tau(|\xi-\eta|^2+|\eta|^2)} \frac{|\xi_1-\eta_1| |\eta_1|}{|\xi-\eta||\eta|} \chi_j^{+}(\xi-\eta) \chi_j^{-}(\eta) d\eta d\tau \nonumber\\
&\leq \frac{C\delta^{2}}{N} \sum_{j=N}^{[\frac{3}{2}N]+1} 2^{2j} 2^{-2j} 2^{-2j} \nonumber\\
&\leq \frac{C\delta^{2}}{N2^{2N}}.
\end{align}
On the estimates of $J_5(\xi,t)$ and $J_6(\xi,t)$, note that
\begin{align*}
|\widehat{G_r}(\xi,\tau)| =\Big| e^{-\tau\mathcal{A}_1(\xi)}\sum_{k=1}^{\infty} \frac{(-\tau\mathcal{A}_2(\xi))^k}{k!} \Big| \leq C e^{-\tau |\xi|^2} \sum_{k=1}^{\infty} \frac{(2^{-2N} 2^{\frac{3N}{2}})^k}{k!}  \leq C 2^{-\frac{N}{2}} e^{-\tau |\xi|^2}
\end{align*}
for all $\xi \in \operatorname{supp} \widehat{u_0^N}$ and all $\tau\in (0,t]$ with $t\sim 2^{-2N}$. Then for $\xi\in E$, we obtain
\begin{align*}
J_5(\xi,t)+J_6(\xi,t) \leq C (2^{-\frac{N}{2}}+2^{-N}) \int_0^t \int_{\R^3} e^{-\tau |\xi-\eta|^2} |\widehat{f^N}(\xi-\eta)| e^{-\tau |\eta|^2} |\widehat{f^N}(\eta)| d\eta d\tau.
\end{align*}
Since it holds that $\sum_{j=-j_0}^{j_0}\widehat{\psi_j}(\xi)=1$ for all $\xi \in E$, Lemma \ref{le2.3}, Lemma \ref{le2.4} and Lemma \ref{le2.5} ensures that
\begin{align}\label{eq4.6}
\|J_5(\xi,t)+J_6(\xi,t)\|_{L^1(E)} & \leq C 2^{-\frac{N}{2}} \Big\| \int_0^t e^{-\tau |\cdot|^2} |\widehat{f^N}(\cdot)| \ast e^{-\tau |\cdot|^2} |\widehat{f^N}(\cdot)| d\tau \Big\|_{L^1(E)}\nonumber\\
& \leq C 2^{-\frac{N}{2}} \Big\{ \sum_{j=-j_0}^{j_0} \Big( \int_0^t \|\widehat{\psi_j}\mathcal{F}[(\mathcal{F}^{-1}[e^{-\tau |\cdot|^2} |\widehat{f^N}|])^2] \|_{L^1} \Big)^2 \Big\}^{\frac{1}{2}} \nonumber\\
& \leq C 2^{-\frac{N}{2}} \|\mathcal{F}^{-1}[e^{-\tau |\cdot|^2} |\widehat{f^N}|]\| \|\mathcal{F}^{-1}[e^{-\tau |\cdot|^2} |\widehat{f^N}|]\| \nonumber\\
& \leq C 2^{-\frac{N}{2}} \|f^N\|^2_{\F^{-1}_{1,2}} \nonumber\\
& \leq C\delta^{2} 2^{-\frac{N}{2}}.
\end{align}
On the estimates of $J_7(\xi,t)$, since it holds that
\begin{align*}
|\widehat{G_r}(\xi,\tau)| =\Big| e^{-\tau\mathcal{A}_1(\xi)}\sum_{k=1}^{\infty} \frac{(-\tau\mathcal{A}_2(\xi))^k}{k!} \Big| \leq C e^{-\tau |\xi|^2} \sum_{k=1}^{\infty} \frac{(2^{-2N})^k}{k!}  \leq C 2^{-2N}
\end{align*}
for all $\xi \in E$ and all $\tau\in (0,t]$ with $t\sim 2^{-2N}$. Then, in the similar way to \eqref{eq4.6}, for $\xi\in E$, we have
\begin{align}\label{eq4.7}
\|J_7(\xi,t)\|_{L^1(E)} & \leq C 2^{-2N} \int_0^t \|\mathcal{F}[G(\tau)f^N\otimes G(\tau)f^N]\|_{L^1}d\tau \nonumber\\
& \leq C 2^{-2N} \|G(\cdot)f^N\| \|G(\cdot)f^N\| \nonumber\\
& \leq C 2^{-2N} \|f^N\|^2_{\F^{-1}_{1,2}} \nonumber\\
& \leq C\delta^{2} 2^{-2N}.
\end{align}
By \eqref{eq4.2}--\eqref{eq4.7}, we see that
\begin{align*}
\|u_2(f^N)(t)\|_{\F^{-1}_{1,r}}  \geq \delta^{2}(C-C 2^{-\frac{N}{2}}-C 2^{-2N}-\frac{C}{N2^N})
\end{align*}
for all $N\in \mathbb{N}$ and $t\sim 2^{-2N}$. Let $N$ large enough, then we have
\begin{align}\label{eq4.8}
\|u_2(f^N)(t)\|_{\F^{-1}_{1,r}}  \geq \frac{C\delta^{2}}{2}.
\end{align}

Now, we consider the fourth component of $\mathcal{F}[A_2(f^N)(t)](\xi)$, which is the first component of $\mathcal{F}[\omega_2(f^N)(t)](\xi)$. Denote $R_{i,j}(\xi,t)$ by the entry in the $i$-th row and $j$-th column of the matrix $R(\xi,t)$, we have
\begin{align}\label{eq4.9}
&|\mathcal{F}[\omega_2(f^N)(t)](\xi)| \nonumber\\
&\geq \Big|\int_0^t \sum_{l=1}^{3} R_{1,l}(\xi,t-\tau) \sum_{k=1}^3 \xi_k (\widehat{G(\tau)f^N})_k * (\widehat{G(\tau)f^N})_{l+3} d\tau\Big| \nonumber\\
&\ \ -\Big|\int_0^t |\widehat{G_r}(\xi,t-\tau)| \mathcal{F}[\tilde{\mathbf{P}}\nabla\cdot (G(\tau)f^N \tilde{\otimes} G(\tau)f^N)](\xi)d\tau\Big| \nonumber\\
&\geq \Big|\int_0^t \sum_{l=1}^{3} R_{1,l}(\xi,t-\tau) \sum_{k=1}^3 \xi_k (\widehat{G_m(\tau)f^N})_k * (\widehat{G_m(\tau)f^N})_{l+3} d\tau\Big| \nonumber\\
&\ \ -\Big|\int_0^t \sum_{l=1}^{3} R_{1,l}(\xi,t-\tau) \sum_{k=1}^3 \xi_k (\widehat{G_r(\tau)f^N})_k * (\widehat{G_m(\tau)f^N})_{l+3} d\tau\Big| \nonumber\\
&\ \ -\Big|\int_0^t \sum_{l=1}^{3} R_{1,l}(\xi,t-\tau) \sum_{k=1}^3 \xi_k (\widehat{G_m(\tau)f^N})_k * (\widehat{G_r(\tau)f^N})_{l+3} d\tau\Big| \nonumber\\
&\ \ -\Big|\int_0^t \sum_{l=1}^{3} R_{1,l}(\xi,t-\tau) \sum_{k=1}^3 \xi_k (\widehat{G_r(\tau)f^N})_k * (\widehat{G_r(\tau)f^N})_{l+3} d\tau\Big| \nonumber\\
&\ \ -\Big|\int_0^t |\widehat{G_r}(\xi,t-\tau)| \mathcal{F}[\tilde{\mathbf{P}}\nabla\cdot (G(\tau)f^N \tilde{\otimes} G(\tau)f^N)](\xi)d\tau\Big| \nonumber\\
& =: K_1(\xi,t)-K_2(\xi,t)-K_3(\xi,t)-K_4(\xi,t)-K_5(\xi,t).
\end{align}
Since $|R_{1,l}(\xi,t-\tau)|\leq C e^{-(t-\tau)|\xi|^2}$, $l=1,2,3$ and $|(\widehat{G_m(\tau)f^N})_k| \leq C e^{-\tau|\xi|^2} |\widehat{f^N}|$, $k=1,2,3,4,5,6$, the similar process for getting $J_5(\xi,t)$ and $J_7(\xi,t)$ gives
\begin{align}\label{eq4.10}
\|K_2(\xi,t)+K_3(\xi,t)+K_4(\xi,t)+K_5(\xi,t)\|_{L^1(E)}\leq C\delta^{2} 2^{-\frac{N}{2}}.
\end{align}
Now we consider the estimate for $K_1(\xi,t)$ with $\xi\in E$.
\begin{align}\label{eq4.11}
K_1(\xi,t)
&= \Big|\int_0^t \sum_{l=1}^{3} R_{1,l}(\xi,t-\tau) \sum_{k=1}^3 \xi_k (e^{-\tau |\xi|^2}\widehat{u_0^N})_k * R_{l,1}(\xi,\tau)  (\widehat{\omega_0^N})_{1} d\tau\Big| \nonumber\\
&\geq \Big|\int_0^t R_{1,1}(\xi,t-\tau) \xi_1 \Big( (e^{-\tau |\xi|^2}\widehat{u_0^N})_1 * R_{1,1}(\xi,\tau)  (\widehat{\omega_0^N})_{1} \Big) d\tau\Big| \nonumber\\
&\ \ -\Big|\int_0^t R_{1,1}(\xi,t-\tau) \xi_2 \Big( (e^{-\tau |\xi|^2}\widehat{u_0^N})_2 * R_{1,1}(\xi,\tau)  (\widehat{\omega_0^N})_{1} \Big) d\tau\Big| \nonumber\\
&\ \ -\Big|\int_0^t \sum_{l=2}^{3} R_{1,l}(\xi,t-\tau) \sum_{k=1}^3 \xi_k (e^{-\tau |\xi|^2}\widehat{u_0^N})_k * R_{l,1}(\xi,\tau)  (\widehat{\omega_0^N})_{1} d\tau\Big| \nonumber\\
&=: K_{11}(\xi,t)-K_{12}(\xi,t)-K_{13}(\xi,t).
\end{align}
On the estimates of $K_{11}(\xi,t)$ for $\xi\in E$, we have
\begin{align*}
K_{11}(\xi,t)& \geq  \Big|-\int_0^t e^{-2(t-\tau)|\xi|^2} \xi_1 \int_{\R^3} e^{-\tau(|\xi-\eta|^2+|\eta|^2)}
\\&\quad \times \Big(\frac{(\xi_2-\eta_2)\eta_2^3}{|\xi-\eta||\eta|^3}+\frac{(\xi_2-\eta_2)^3\eta_2}{|\xi-\eta|^3|\eta|}\Big) \frac{\delta^{2}}{N} \sum_{j=N}^{[\frac{3}{2}N]+1} 2^{2j} \chi_j^{+}(\xi-\eta) \chi_j^{-}(\eta) d\eta d\tau\Big|.
\end{align*}
Since
\begin{align*}
 -1 \leq \frac{(\xi_2-\eta_2)\eta_2^3}{|\xi-\eta||\eta|^3}+\frac{(\xi_2-\eta_2)^3\eta_2}{|\xi-\eta|^3|\eta|} \leq -\frac{1}{256}
\end{align*}
for all $\eta \in \operatorname{supp} \chi_j^{-}$ with $\xi-\eta \in \operatorname{supp} \chi_j^{+}$, or all $\eta \in \operatorname{supp} \chi_j^{+}$ with $\xi-\eta \in \operatorname{supp} \chi_j^{-}$, we obtain
\begin{align*}
K_{11}(\xi,t) &\geq \frac{C\delta^{2}}{N} \sum_{j=N}^{[\frac{3}{2}N]+1} 2^{2j} \int_0^t \int_{\R^3} e^{-\tau(|\xi-\eta|^2+|\eta|^2)} \chi_j^{+}(\xi-\eta) \chi_j^{-}(\eta) d\eta d\tau \\
&\geq \frac{C\delta^{2}}{N} \sum_{j=N}^{[\frac{3}{2}N]+1} 2^{2j} 2^{-2j} (1-e^{8 2^{2j}t}) \\
&\geq C\delta^{2}
\end{align*}
for all $\xi\in E$ and all $t\sim 2^{-2N}$, which yields that
\begin{align}\label{eq4.12}
\|K_{11}(\xi,t)\|_{L^1(E)}\geq C\delta^{2}.
\end{align}
On the estimates of $K_{12}(\xi,t)$ for $\xi\in E$, we have
\begin{align}\label{eq4.13}
K_{12}(\xi,t)& \leq C\delta^{2} \Big|\int_0^t e^{-(t-\tau)|\xi|^2} \xi_1 \int_{\R^3} e^{-\tau(|\xi-\eta|^2+|\eta|^2)}
\nonumber\\&\quad \times \Big(\frac{(\xi_1-\eta_1)\eta_2^3}{|\xi-\eta||\eta|^3}+\frac{(\xi_2-\eta_2)^3\eta_1}{|\xi-\eta|^3|\eta|}\Big) \frac{1}{N} \sum_{j=N}^{[\frac{3}{2}N]+1} 2^{2j} \chi_j^{+}(\xi-\eta) \chi_j^{-}(\eta) d\eta d\tau\Big| \nonumber\\
&\leq \frac{C\delta^{2}}{N} \sum_{j=N}^{[\frac{3}{2}N]+1} 2^{2j} 2^{-2j} 2^{-j} \nonumber\\
&\leq \frac{C\delta^{2}}{N2^N}.
\end{align}
On the estimates of $K_{13}(\xi,t)$ for $\xi\in E$, we have
\begin{align}\label{eq4.14}
K_{13}(\xi,t)& \leq C\delta^{2} \Big|\int_0^t e^{-(t-\tau)|\xi|^2} \xi_1 \int_{\R^3} e^{-\tau(|\xi-\eta|^2+|\eta|^2)}
\nonumber\\&\quad \times \Big(\frac{(\xi_1-\eta_1)(\xi_2-\eta_2)^2\eta_2}{|\xi-\eta|^3|\eta|}+\frac{(\xi_2-\eta_2)\eta_1\eta_2^2}{|\xi-\eta||\eta|^3}\Big) \frac{1}{N} \sum_{j=N}^{[\frac{3}{2}N]+1} 2^{2j} \chi_j^{+}(\xi-\eta) \chi_j^{-}(\eta) d\eta d\tau\Big| \nonumber\\
&\leq \frac{C\delta^{2}}{N} \sum_{j=N}^{[\frac{3}{2}N]+1} 2^{2j} 2^{-2j} 2^{-j} \nonumber\\
&\leq \frac{C\delta^{2}}{N2^N}.
\end{align}
By \eqref{eq4.9}--\eqref{eq4.14}, we see that
\begin{align*}
\|\omega_2(f^N)(t)\|_{\F^{-1}_{1,r}}  \geq \delta^{2}(C-C 2^{-\frac{N}{2}}-\frac{C}{N2^N})
\end{align*}
for all $N\in \mathbb{N}$ and $t\sim 2^{-2N}$. Let $N$ large enough, then we have
\begin{align}\label{eq4.15}
\|\omega_2(f^N)(t)\|_{\F^{-1}_{1,r}}  \geq \frac{C\delta^{2}}{2}.
\end{align}
%In addition, by Lemma \ref{le2.3}, Lemma \ref{le2.5} and Lemma \ref{le2.6}, we obtain
%\begin{align*}
%\sum_{k=3}^{\infty} \|A_k(u_0^N)(t)\|_{\F^{-1}_{1,2}} \leq C \delta^3.
%\end{align*}
Hence, applying Lemma \ref{le2.2}, Lemma \ref{le2.4}, Lemma \ref{le2.5} and putting \eqref{eq4.1} and \eqref{eq4.8} together yields
\begin{align*}
\|u[f^N]\|_{\F^{-1}_{1,r}}  &\geq \|u_2(f^N)(t)\|_{\F^{-1}_{1,r}}-\|A_1(u_0^N)(t)\|_{\F^{-1}_{1,r}}-\sum_{k=3}^{\infty} \|A_k(u_0^N)(t)\|_{\F^{-1}_{1,2}} \\
&\geq C_1 \delta^2 -C_2\delta N^{-\frac{1}{2}+\frac{1}{r}}-C_2 \delta^3 \\
&\geq \frac{C_1}{2} \delta^2,
\end{align*}
here we choose $\delta=\min\{\frac{1}{2},\frac{C_1}{4C_2}\}$ and $N>\max\{1000,\frac{4C_2}{\delta C_1}^{\frac{2r}{r-2}}\}$.
Similarly, we have
\begin{align*}
\|\omega[f^N]\|_{\F^{-1}_{1,r}}  &\geq \|\omega_2(f^N)(t)\|_{\F^{-1}_{1,r}}-\|A_1(\omega_0^N)(t)\|_{\F^{-1}_{1,r}}-\sum_{k=3}^{\infty} \|A_k(\omega_0^N)(t)\|_{\F^{-1}_{1,2}} \\
&\geq C_1 \delta^2 -C_2\delta N^{-\frac{1}{2}+\frac{1}{r}}-C_2 \delta^3 \\
&\geq \frac{C_1}{2} \delta^2.
\end{align*}
This completes the proof of Theorem \ref{th1.5}.

  {\em The proof of Theorem \ref{th1.6}}:\ \ Next, we prove ill-posedness for system \eqref{eq1.1} in $\B^{-1}_{\infty,r}(\R^3)$ with
  $r\in (2,\infty]$. Similarly to the proof in the case of $\F^{-1}_{1,r}$, we use the same sequence of initial data. By $\F^{-1}_{1,r}(\R^3) \hookrightarrow \B^{-1}_{\infty,r}(\R^3)$, we have
\begin{align*}
\|f^N\|_{\B^{-1}_{\infty,r}} \leq C\delta N^{\frac{1}{r}-\frac{1}{2}}.
\end{align*}
Define
\begin{align*}
\bar{E} := \{\xi\in \R^3\mid \frac{1}{20}\leq \xi_1 \leq \frac{11}{10},\ |\xi|\leq \frac{11}{10}\},
\end{align*}
then $E \subset\subset \bar{E}$.
By \eqref{eq4.2}, it is easy to check that
\begin{align*}
&\quad\int_0^t e^{-(t-\tau)|\xi|^2} (1-\frac{\xi_1^2}{|\xi|^2}) \xi_1 (e^{-\tau|\xi|^2} \widehat{u_0^N})_1 * (e^{-\tau|\xi|^2} \widehat{u_0^N})_1 d\tau \\
&=-2\int_0^t e^{-(t-\tau)|\xi|^2} (1-\frac{\xi_1^2}{|\xi|^2}) \xi_1 \int_{\R^3} e^{-\tau(|\xi-\eta|^2+|\eta|^2)}
\\&\quad\quad\quad\times \frac{(\xi_2-\eta_2)\eta_2}{|\xi-\eta||\eta|} \frac{\delta^2}{N} \sum_{j=N}^{[\frac{3}{2}N]+1} 2^{2j} \chi_j^{+}(\xi-\eta) \chi_j^{-}(\eta) d\eta d\tau\\
& \geq 0,
\end{align*}
then $\mathcal{F}[u_2(f^N)(t)](\xi)$ is a nonnegative locally integrable function in $\bar{E}$. Let $\phi$ be a nonnegative functions in $\mathcal{S}(\mathbb{R}^3)$ supported in $\bar{E}$ such that $\phi(\xi)=1$ for all $\xi \in E$. Thus
\begin{align*}
\|u_2(f^N)(t)\|_{\B^{-1}_{\infty,r}} &\geq C \|\mathcal{F}^{-1}[\phi(\xi)] \ast \mathcal{F}^{-1}[\mathcal{F}[u_2(f^N)(t)](\xi)]\|_{L^\infty(\R^3)} \\
&\geq C \|\mathcal{F}^{-1}[\phi(\xi)\mathcal{F}[u_2(f^N)(t)](\xi)]\|_{L^\infty(\R^3)} \\
&\geq C \|\mathcal{F}[u_2(f^N)(t)](\xi)\|_{L^1(E)} \\
&\geq C_1\delta^2.
\end{align*}
Here we apply the following proposition: Let $f\in \mathcal{S}'(\R^3)$. If $\hat{f}(\xi)$ is a nonnegative locally integrable function, then $f$ is bounded if and only if $\hat{f}$ is integrable, and in this case $\|f\|_{L^\infty}=(2\pi)^{-3}\|\hat{f}\|_{L^1}$. Therefore, we obtain
\begin{align*}
\|u[f^N]\|_{\B^{-1}_{\infty,r}}  &\geq \|u_2(f^N)(t)\|_{\B^{-1}_{\infty,r}}-\|A_1(u_0^N)(t)\|_{\B^{-1}_{\infty,r}}-\sum_{k=3}^{\infty} \|A_k(u_0^N)(t)\|_{\F^{-1}_{1,2}} \\
&\geq C_1 \delta^2 -C_2\delta N^{-\frac{1}{2}+\frac{1}{r}}-C_2 \delta^3 \\
&\geq \frac{C_1}{2} \delta^2.
\end{align*}
Similarly, by \eqref{eq4.11}, we have
\begin{align*}
&\quad\int_0^t R_{1,1}(\xi,t-\tau) \xi_1 \Big( (e^{-\tau |\xi|^2}\widehat{u_0^N})_1 * R_{1,1}(\xi,\tau)  (\widehat{\omega_0^N})_{1} \Big) d\tau \\
&=-\int_0^t e^{-2(t-\tau)|\xi|^2} \xi_1 \int_{\R^3} e^{-\tau(|\xi-\eta|^2+|\eta|^2)}
\\&\quad \times \Big(\frac{(\xi_2-\eta_2)\eta_2^3}{|\xi-\eta||\eta|^3}+\frac{(\xi_2-\eta_2)^3\eta_2}{|\xi-\eta|^3|\eta|}\Big) \frac{\delta^{2}}{N} \sum_{j=N}^{[\frac{3}{2}N]+1} 2^{2j} \chi_j^{+}(\xi-\eta) \chi_j^{-}(\eta) d\eta d\tau \\
&\geq 0,
\end{align*}
then $\mathcal{F}[\omega_2(f^N)(t)](\xi)$ is a nonnegative locally integrable function in $\bar{E}$. Thus
\begin{align*}
\|\omega_2(f^N)(t)\|_{\B^{-1}_{\infty,r}} \geq C \|\mathcal{F}[\omega_2(f^N)(t)](\xi)\|_{L^1(E)} \geq C_1\delta^2,
\end{align*}
which implies that $\|\omega[f^N]\|_{\B^{-1}_{\infty,r}}\geq \frac{C_1}{2} \delta^2$. This completes the proof of Theorem \ref{th1.6}.

\medskip
\medskip

\noindent\textbf{Acknowledgments.} The author is sincerely to acknowledge his gratefulness to Professor Shangbin Cui for his warm guidance. This work is supported by the National Natural Science Foundation of China under the grant number 11571381.

\bibliographystyle{abbrv}

\end{document}